\documentclass{amsart}
\usepackage[matrix,arrow]{xy}
\usepackage{amssymb}
\usepackage{mathrsfs}

\numberwithin{equation}{section}

\newtheorem{thm}{Theorem}[section]
\newtheorem{prop}[thm]{Proposition}
\newtheorem{lem}[thm]{Lemma}
\newtheorem{cor}[thm]{Corollary}
\theoremstyle{definition}

\theoremstyle{remark}
\newtheorem{rem}[thm]{Remark}

\newcommand{\Z}{\mathbb{Z}}
\newcommand{\C}{\mathbb{C}}

\newcommand{\R}{\mathbb{R}}

\newcommand{\Perv}{\mathrm{Perv}}
\newcommand{\HC}{\mathcal{HC}}

\DeclareMathOperator{\Hom}{Hom}
\DeclareMathOperator{\Ann}{Ann}

\DeclareMathOperator{\Ad}{Ad}
\DeclareMathOperator{\Stab}{Stab}
\DeclareMathOperator{\Lie}{Lie}
\DeclareMathOperator{\PGL}{PGL}
\DeclareMathOperator{\Spec}{Spec}
\newcommand{\wtsp}{\Gamma}
\newcommand{\wt}{\mathrm{wt}}

\newcommand{\Ker}{\mathrm{Ker}}
\newcommand{\Cbatu}{\mathbb{G}_{\mathit{m}}}
\newcommand{\Caff}{\mathbb{A}}
\newcommand*{\Kat}[1]{\mathrm{Kat}_{#1}}

\title{Jacquet functor and De Concini-Procesi compactification}
\author{Noriyuki Abe}
\author{Yoichi Mieda}
\address[Noriyuki Abe]{Creative Research Institution, Hokkaido University, N21 W10, Kita-ku Sapporo 001-0021, Japan}
\email{abenori@math.sci.hokudai.ac.jp}
\address[Yoichi Mieda]{Faculty of Mathematics, Kyushu University, 744 Motooka, Nishi-ku, Fukuoka, 819-0395 Japan}
\email{mieda@math.kyushu-u.ac.jp}

\begin{document}
\begin{abstract}
We give a geometric realization of the Jacquet functor using a deformation of De Concini-Procesi compactification.
\end{abstract}

\maketitle
\section{Introduction}
The symmetric variety is used from a long time ago, when representations of a real reductive group are studied by the analytic way.
On the other hand, when they are studied by the algebraic way, due to the localization theorem of Beilinson-Bernstein~\cite{MR610137}, the flag variety is often used.
However, the geometry of the symmetric space is richer than that of the flag variety.
For example, the symmetric space has a boundary and one can take a ``limit'' to this boundary (cf.~\cite{MR485861}).

Fortunately, the localization theorem gives a way to realize representations as geometric objects on the symmetric variety $G/K$, where $G$ (resp.~$K$) is the complexification of a real reductive group $G_\R$ (resp.~a maximal compact subgroup $K_\R$ of $G_\R$).
However, as far as the authors know, little is studied by such a way.
In this paper, we use the symmetric variety and try to take a ``limit'' of such a geometric object.
The limit should become the Jacquet module~\cite{MR562655} since the Jacquet module describes the asymptotic behavior of matrix coefficients~\cite{MR705884}.
In the $p$-adic case, similar results can be found in a work of Schneider-Stuhler~\cite{MR1471867}.
They realize the Jacquet module on the boundary of the Borel-Serre compactification of the Bruhat-Tits building.
The vertices of the Bruhat-Tits building for a $p$-adic semisimple group $G$ are in bijection with
a union of sets of the form $G/K$, where $K$ is a maximal compact subgroup of $G$.
So it can be regarded as an analogue of the symmetric variety.
In this paper, we realize the Jacquet module by taking a limit on the symmetric space.
Notice that, if you use the flag variety instead of the symmetric space, a realization of the Jacquet module has already been given by Emerton-Nadler-Vilonen~\cite{MR2096674}.

We state our main results.
Assume that $G_\R$ is of adjoint type.
Let $G_\R = K_\R A_\R N_\R$ be an Iwasawa decomposition, and $M_\R$ the centralizer of $A_\R$ in $K_\R$.
Then $P_\R = M_\R A_\R N_\R$ is a Langlands decomposition of a minimal parabolic subgroup.
We use lower-case fraktur letters to denote the corresponding Lie algebras and omit the subscripts ``$\R$'' to denote complexifications.
Let $X$ be the De Concini-Procesi compactification of $G/K$~\cite{MR718125}.
The $G$-orbit of $X$ is parameterized by a subset of $\Pi$, where $\Pi\subset\Hom_\R(\mathfrak{a}_\R,\R)$ is the set of simple restricted roots.
Consider the closures of the codimension $1$ orbits $\{Y_\alpha\}_{\alpha\in\Pi}$.
Then one can construct the variety $\mathcal{X}$ over $\Caff^\Pi$ by iterating the deformation to the
normal cone (see Section~\ref{sec:Deformation to normal cone}).
The subvariety $Y_\alpha\subset X$ defines the subvariety $\mathcal{Y}_\alpha\subset \mathcal{X}$.
Put $\mathcal{Z} = \mathcal{X}\setminus\bigcup_{\alpha\in\Pi}\mathcal{Y}_\alpha$.
This is the variety which we will use.
An important property of this variety is given by the following proposition.
For $\Theta\subset\Pi$, let $P_{\Theta,\R} = M_{\Theta,\R}A_{\Theta,\R}N_{\Theta,\R}$ be a parabolic subgroup of $G_\R$ corresponding to $\Theta$.
Put $K_\Theta = M_\Theta\cap K$.
Let $f_\mathcal{Z}\colon \mathcal{Z}\to\Caff^\Pi$ be the canonical morphism.
\begin{prop}[Lemma~\ref{lem:stabilizer of x_Theta}, \ref{lem:transitive on fiber of Z}]\label{prop:structure of Z}
For $\Theta\subset\Pi$, we have $f_{\mathcal{Z}}^{-1}((\Cbatu)^{\Theta}\times \{0\}^{\Pi\setminus\Theta}) \simeq G/K_\Theta N_\Theta\times (\Cbatu)^{\Theta}$.
\end{prop}

Since $X$ has many orbits, it is natural to consider the ``partial'' Jacquet modules.
Let $\HC_{\Theta,\rho}$ be the category of finitely generated $(\mathfrak{g},K_\Theta N_\Theta)$-modules with the same infinitesimal characters as that of the trivial representation.
For $\Theta_2\subset\Theta_1\subset\Pi$ and $V\in \HC_{\Theta_1,\rho}$, put
\[
	J_{\Theta_2,\Theta_1}(V) = \{v\in \varprojlim_k V/(\mathfrak{m}_{\Theta_1}\cap\overline{\mathfrak{n}}_{\Theta_2})^kV\mid \text{$\mathfrak{n}_{\Theta_2}^lv = 0$ for some $l$}\}
\]
where $\overline{\mathfrak{n}}_{\Theta_2}$ is the nilradical of the parabolic subalgebra opposite to $\mathfrak{p}_{\Theta_2}$.
Then we can prove that $J_{\Theta_2,\Theta_1}(V)\in \HC_{\Theta_2,\rho}$ (Proposition~\ref{prop:image of Jacquet functor}).
The usual Jacquet module is $J_{\emptyset,\Pi}$.
In general, let $\Perv_H(Z)$ be the category of $H$-equivariant perverse sheaves on $Z$ for an algebraic group $H$ and a $H$-variety $Z$.
Let $B$ be a Borel subgroup of $G$.
The Beilinson-Bernstein correspondence and the Riemann-Hilbert correspondence give an equivalence of categories $\HC_{\Theta,\rho}\simeq\Perv_{K_\Theta N_\Theta}(G/B)$. The latter category is obviously equivalent
to the category $\Perv_{B}(G/K_\Theta N_\Theta)$, and thus we obtain an equivalence 
$\HC_{\Theta,\rho}\simeq\Perv_{B}(G/K_\Theta N_\Theta)$.

Now we give our main theorem.
Fix $\Theta_2\subset\Theta_1\subset\Pi$ and for each $\alpha\in\Theta_1\setminus\Theta_2$, take $n_\alpha\in\Z_{>0}$.
Define $\nu\colon \Caff^1\to\Caff^\Pi$ by $\nu(t) = (0^{\Pi\setminus\Theta_1},(t^{n_\alpha})_{\alpha\in\Theta_1\setminus\Theta_2},0^{\Theta_2})$.
Put $f_\nu\colon \mathcal{Z}_\nu = \mathcal{Z}\times_{\Caff^\Pi}\Caff^1\to\Caff^1$.
The by Proposition~\ref{prop:structure of Z}, we have $f_\nu^{-1}(\Cbatu)\simeq G/K_{\Theta_1} N_{\Theta_1}\times\Cbatu$ and $f_\nu^{-1}(0) \simeq G/K_{\Theta_2} N_{\Theta_2}$.
Let $p_\nu\colon f_\nu^{-1}(\Cbatu)\simeq G/K_{\Theta_1} N_{\Theta_1}\times\Cbatu\to G/K_{\Theta_1} N_{\Theta_1}$ be the natural projection and $R\psi\colon\Perv(f_\nu^{-1}(\Cbatu))\to \Perv(f_\nu^{-1}(0))$ be the nearby cycle functor.
Define the functor $\Kat{\nu}\colon \Perv(G/K_{\Theta_1}N_{\Theta_1})\to \Perv(G/K_{\Theta_2}N_{\Theta_2})$ by $\Kat{\nu} = R\psi\circ p_\nu^*$.
\begin{thm}
We have the following commutative diagram:
\[
\xymatrix{
\HC_{\Theta_1,\rho}\ar[r]^{J_{\Theta_2,\Theta_1}}\ar@{<->}[d]_\simeq & \HC_{\Theta_2,\rho}\ar@{<->}[d]^\simeq\\
\Perv_B(G/K_{\Theta_1}N_{\Theta_1})\ar[r]^{\Kat{\nu}} & \Perv_B(G/K_{\Theta_2}N_{\Theta_2}).
}
\]
\end{thm}

We summarize the contents of this paper.
In Section~\ref{sec:Deformation to normal cone}, we give preliminaries on the deformation to the normal cone.
The definition and the properties of the partial Jacquet functor $J_{\Theta_2,\Theta_1}$
are given in Section~\ref{sec:The Jacquet functors}.
We review the theorem of Emerton-Nadler-Vilonen in Section~\ref{sec:The geometric Jacquet functor}.
We will use their result to prove our theorem.
We finish a proof of the main theorem in Section~\ref{sec:Symmetric spaces}.

\section*{Acknowledgment}
We thank Syu Kato for giving us the construction when $G_\R = \PGL_2(\R)$.

\section{Deformation to normal cone}\label{sec:Deformation to normal cone}
Let $X$ be a scheme of finite type over $\C$ and $Y$ its closed subscheme.
Then we can construct a family $f\colon \mathcal{X}\to \Caff^1$, called the deformation to the normal cone,
that satisfies the following:
\begin{itemize}
 \item over $\Cbatu$, it is a constant family $X\times\Cbatu\to \Cbatu$, and
 \item the fiber $f^{-1}(0)$ at $0$ is isomorphic to the normal cone $C_Y(X)$.
\end{itemize}
Recall that, if we write $\mathcal{I}$ for the defining ideal of $Y\subset X$, 
the normal cone $C_Y(X)$ is the scheme $\mathop{\mathbf{Spec}} \bigoplus_{k=0}^\infty\mathcal{I}^k/\mathcal{I}^{k+1}$ over $Y$.
If $X$ and $Y$ are smooth over $\C$, $C_Y(X)$ is isomorphic to the normal bundle $T_Y(X)$. 

Let us recall briefly its construction. For more detail, see \cite[Chapter 5]{MR1644323}.
Let $\widetilde{\mathcal{X}}$ be the blow-up of $X\times \Caff^1$
along $Y\times \{0\}$. Then, by the universal property of the blow-up, we have a natural morphism
$X\times \{0\}\to \widetilde{\mathcal{X}}$, which is a closed immersion. We define $\mathcal{X}$ as the complement
of its image in $\widetilde{\mathcal{X}}$, and $f\colon \mathcal{X}\to \Caff^1$ as the composite of
$\mathcal{X}\hookrightarrow \widetilde{\mathcal{X}}\to X\times\Caff^1\to \Caff^1$.
If $X$ is an affine scheme $\Spec A$, we may describe $\mathcal{X}$ more explicitly as follows.
Let $I$ be the defining ideal of $Y\subset X$. 
Then $\mathcal{X}=\Spec \bigoplus_{n\in \Z}I^{-n}T^n$, where $T$ is an indeterminate such that
$\Caff^1=\Spec \C[T]$. Note that we set $I^n=A$ for a negative integer $n$, and regard
$\bigoplus_{n\in \Z}I^{-n}T^n$ as a subring of the Laurent polynomial ring $A[T^{\pm 1}]$.

To any subscheme $Z$ of $X$, we can attach a subscheme $\mathcal{Z}$ of $\mathcal{X}$
such that $\mathcal{Z}\to \Caff^1$ is the deformation to the normal cone with respect to
$Y\cap Z\subset Z$. If $Z$ is open in $X$, then $\mathcal{Z}$ is simply the inverse image
of $Z\times\Caff^1$ under $\mathcal{X}\to X\times \Caff^1$. On the other hand,
if $Z$ is closed in $X$, then $\mathcal{Z}$ is the strict transform of
$Z\times\Caff^1\subset X\times \Caff^1$ in $\mathcal{X}$;
namely, $\mathcal{Z}$ is the closure of $f^{-1}(Z\times \Cbatu)$ in $\mathcal{X}$.

For our purpose, iteration of this construction is important. Now let $X$ be a scheme which is smooth of
finite type over $\C$, $Y$ its effective divisor, and $\bigcup_{i=1}^lY_i$ the irreducible decomposition
of $Y$.
Assume that $Y$ is a strict normal crossing divisor.
Namely, for each subset $\Theta\subset \{1,\ldots, l\}$, we assume that $Y_\Theta=\bigcap_{i\in \Theta}Y_i$ is smooth over $\C$.
It is equivalent to saying that $Y\subset X$ is \'etale locally isomorphic to 
$(T_1\cdots T_l=0)\subset \Caff^n=\Spec \C[T_1,\ldots,T_n]$ and every irreducible component of $Y$ 
is smooth over $\C$.
Under this setting, let $\mathcal{X}^{(1)}\to \Caff^1$ be the deformation to the normal cone with respect to
$Y_1\subset X$. For each $i$, the closed subscheme $Y_i$ of $X$ induces a closed subscheme $\mathcal{Y}_i^{(1)}$
of $\mathcal{X}^{(1)}$. Next, consider the deformation to the normal cone $\mathcal{X}^{(2)}\to \Caff^1$
with respect to $\mathcal{Y}_2^{(1)}\subset \mathcal{X}^{(1)}$ and closed subschemes $\mathcal{Y}^{(2)}_i$.
Inductively, we can define a family $\mathcal{X}^{(k)}\to \Caff^1$ and 
closed subschemes $\mathcal{Y}^{(k)}_i$ of $\mathcal{X}^{(k)}$.
Recall that, by construction, $\mathcal{X}^{(k)}$ is equipped with a natural structure morphism
$\mathcal{X}^{(k)}\to \mathcal{X}^{(k-1)}\times \Caff^1$, where $\mathcal{X}^{(k)}\to\Caff^1$ is the composite
of it with the second projection (here we put $\mathcal{X}^{(0)}=X$). Therefore, we get a natural morphism 
$\pi_k\colon \mathcal{X}^{(k)}\to X\times \Caff^k$ and $f_k=\mathrm{pr}_2\circ \pi\colon \mathcal{X}^{(k)}\to \Caff^k$. If $k=l$, we simply write $\mathcal{X}$, $\mathcal{Y}_i$, $\pi$, $f$
for $\mathcal{X}^{(l)}$, $\mathcal{Y}_i^{(l)}$, $\pi_l$, $f_l$, respectively.

First let us consider \'etale locally. Assume that $X=\Caff^n=\Spec \C[S_1,\ldots,S_n]$ and
$Y_i$ is given by the equation $S_i=0$ for each $i$. Then we have
\begin{align*}
 \mathcal{X}^{(1)}&=\Spec \C\Bigl[\frac{S_1}{T_1},S_2,\ldots,S_n,T_1\Bigr],\quad \mathcal{Y}^{(1)}_i\colon 
 \frac{S_i}{T_i}=0\ (i=1),\ S_i=0\ (i\ge 2),\\
 \mathcal{X}^{(2)}&=\Spec \C\Bigl[\frac{S_1}{T_1},\frac{S_2}{T_2},S_3,\ldots,S_n,T_1,T_2\Bigr],\quad
 \mathcal{Y}^{(2)}_i\colon \frac{S_i}{T_i}=0\ (i\le 2),\ S_i=0\ (i\ge 3),\\
 &\vdots\\
 \mathcal{X}^{(l)}&=\Spec \C\Bigl[\frac{S_1}{T_1},\ldots,\frac{S_l}{T_l},S_{l+1},\ldots,S_n,T_1,\ldots,T_l\Bigr],
 \quad \mathcal{Y}^{(l)}_i\colon \frac{S_i}{T_i}=0.
\end{align*}

This computation can be generalized to the case where $X$ is affine:

\begin{lem}\label{lem:affine calculation of deformation}
 Assume that $X$ is an affine scheme $\Spec A$. Let $I_i$ be the defining ideal of $Y_i$. Then, we have
 \[
  \mathcal{X}=\Spec \bigoplus_{\underline{n}\in\Z^l}I^{-\underline{n}}T^{\underline{n}},
 \]
 where $I^{-\underline{n}}=I_1^{-n_1}\cdots I_l^{-n_l}$ and $T^{\underline{n}}=T_1^{n_1}\cdots T_l^{n_l}$
 for $\underline{n}=(n_1,\ldots,n_l)\in\Z^l$.
\end{lem}

\begin{proof}
 As explained above, we have $\mathcal{X}^{(1)}=\Spec B^{(1)}$
 for $B^{(1)}=\bigoplus_{n_1\in\Z}I_1^{-n_1}T_1^{n_1}$.
 The local calculation tells us that the defining ideal of $\mathcal{Y}_2^{(1)}$ is $I_2B^{(1)}$.
 Thus we have $\mathcal{X}^{(2)}=\Spec B^{(2)}$ for 
 $B^{(2)}=\bigoplus_{n_2\in\Z}I_2^{-n_2}B^{(1)}T_2^{n_2}=\bigoplus_{(n_1,n_2)\in\Z^2}I_1^{-n_1}I_2^{-n_2}T_1^{n_1}T_2^{n_2}$.
 We can proceed similarly to obtain the desired formula.
\end{proof}

\begin{rem}
 By the lemma above, we know that $\mathcal{X}$ is independent of the labeling of $Y_1,\ldots,Y_l$.
\end{rem}

For each subset $\Theta$ of $\{1,\ldots,l\}$, set $\mathcal{X}_\Theta=f^{-1}((\Cbatu)^\Theta\times \{0\}^{\Theta^c})$,
where $\Theta^c=\{1,\ldots,l\}\setminus \Theta$. Recall that we put $Y_\Theta=\bigcap_{i\in \Theta}Y_i$.

\begin{lem}\label{lem:deformation to normal cone}
 We have $\mathcal{X}_\Theta\simeq C_{Y_{\Theta^c}}(X)\times (\Cbatu)^\Theta$.
\end{lem}

\begin{proof}
 We may assume that $X$ is an affine scheme $\Spec A$. Let $I_i$ be the defining ideal of $Y_i$ and
 put $B=\bigoplus_{\underline{n}\in\Z^l}I^{-\underline{n}}T^{\underline{n}}$. 
 By Lemma \ref{lem:affine calculation of deformation}, we have
 \[
  \mathcal{X}_\Theta=\Spec B[T_i^{-1}\,\vert\, i\in \Theta]/(T_i\,\vert\,i\in \Theta^c),
 \]
 where $(T_i\,\vert\,i\in \Theta^c)$ is the ideal generated by $T_i$ for $i\in \Theta^c$.
 Note that $B[T_i^{-1}\,\vert\, i\in \Theta]$ is isomorphic to
 $(\bigoplus_{\underline{n}\in\Z^{\Theta^c}}I^{-\underline{n}}T^{\underline{n}})\otimes_\C\C[T_i^{\pm 1}\,\vert\,i\in \Theta]$.
 Therefore, by replacing $\{Y_1,\ldots,Y_l\}$ with $\{Y_i\mid i\in \Theta^c\}$,
 we may assume that $\Theta=\emptyset$.
 
 Put $J=I_1+\cdots+I_l$. By the definition, we have
 \[
  \mathcal{X}_{\emptyset}=\Spec B/(T_1,\ldots,T_l)=\Spec \bigoplus_{\underline{n}\in (\Z_{\ge 0})^l}(I^{\underline{n}}/JI^{\underline{n}})T^{-\underline{n}}=\Spec \bigoplus_{\underline{n}\in (\Z_{\ge 0})^l}I^{\underline{n}}/JI^{\underline{n}}.
 \]
 On the other hand, we have $C_{Y_{\{1,\ldots,l\}}}(X)=\Spec \bigoplus_{k=0}^\infty J^k/J^{k+1}$.
 Therefore we have a natural morphism $C_{Y_{\{1,\ldots,l\}}}(X)\to \mathcal{X}_{\emptyset}$
 by sending $I^{\underline{n}}/JI^{\underline{n}}$ to $J^k/J^{k+1}$ where $k=n_1+\cdots+n_l$.
 By \'etale local calculation, it is easily seen that this morphism is an isomorphism.
\end{proof}

\begin{lem}\label{lem:local calculation of deformation}
 Assume that $X=X'\times\Caff^l$ and $Y_i=\{(x,(c_j))\in X\mid c_i=0\}$.
 \begin{enumerate}
  \item We have an isomorphism $\mathcal{X}\simeq X'\times \Caff^l\times\Caff^l$ under which
	$\pi\colon \mathcal{X}\to X\times \Caff^l=X'\times\Caff^l\times\Caff^l$ is given by
	$(x,(d_i),(t_i))\mapsto (x,(d_it_i),(t_i))$.
	\label{enum:local calculatio of deformation:description}
  \item The projection $\mathcal{X}_\Theta\simeq C_{Y_{\Theta^c}}(X)\times(\Cbatu)^\Theta\to C_{Y_{\Theta^c}}(X)\to Y_{\Theta^c}$ is
	given by $(x,(d_i),(t_i))\mapsto (x,(d_it_i))$.
	\label{enum:local calculatio of deformation:projection}
  \item Let $G$ be an algebraic group over $\C$.
	Assume that we are given an action of $G$ on $X'$ and characters
	$\chi_i\colon G\to \mathbb{G}_m$. These induce an action of $G$ on $X$ by
	$(x,(c_i))\to (gx,(\chi_i(g)c_i))$ which preserves $Y_i$. 
	Then, the induced action on $\mathcal{X}\simeq X'\times \Caff^l\times\Caff^l$ is given by
	$(x,(d_i),(t_i))\mapsto (gx,(\chi_i(g)d_i),(t_i))$.
	\label{enum:local calculatio of deformation:group action}
 \end{enumerate}
\end{lem}

\begin{proof}
 (1) Since our construction clearly commutes with a smooth base change, we may assume that $X'=\Spec \C$.
 Then the local calculation above gives us the desired isomorphism (we have only to take $n=l$).

 (2) In the same way as in (1), we may assume that $X'=\Spec \C$ and use the local calculation. 

 (3) Since $\pi\colon \mathcal{X}\to X\times\Caff^l$ is $G$-equivariant,
  for $(x,(d_i),(t_i))\in \mathcal{X}$, we have
 $\pi(g(x,(d_i),(t_i)))=g(x,(d_it_i),(t_i))=(gx,(\chi_i(g)d_it_i),(t_i))=\pi(gx,(\chi_i(g)d_i),(t_i))$.
 On the other hand, $\pi$ is an isomorphism over the open subset $X\times (\Cbatu)^l\subset X\times \Caff^l$.
 Therefore, on the dense open subset $\pi^{-1}(X\times (\Cbatu)^l)$ of $\mathcal{X}$,
 the action of $G$ is given by $(x,(d_i),(t_i))\mapsto (gx,(\chi_i(g)d_i),(t_i))$.
 Hence it is given by the same formula over the whole $\mathcal{X}$.
\end{proof}

\section{The Jacquet functors}\label{sec:The Jacquet functors}
In this section, we recall some preliminaries on the Jacquet modules, which are well-known. (For example, some of them are proved in \cite{MR705884,MR929683}.)
However, we give proofs for the sake of completeness.

Let $G_\R$ be a connected reductive linear algebraic group over $\R$, $G_\R = K_\R A_\R N_\R$ an Iwasawa decomposition, and $M_\R$ the centralizer of $A_\R$ in $K_\R$.
Then $P_\R = M_\R A_\R N_\R$ is a Langlands decomposition of a minimal parabolic subgroup.
We use lower-case fraktur letters to denote the corresponding Lie algebras and omit the subscripts ``$\R$'' to denote complexifications.
Fix a Cartan involution $\theta$ such that $K = \{g\in G\mid \theta(g) = g\}$.
Let $\Sigma$ be the restricted root system for $(\mathfrak{g},\mathfrak{a})$ and $\Sigma^+$ the positive system corresponding to $\mathfrak{n}$.
Then $\Sigma^+$ determines the set of simple roots $\Pi\subset\Sigma^+$.
As usual, the universal enveloping algebra of $\mathfrak{g}$ is denoted by $U(\mathfrak{g})$ and the center of $U(\mathfrak{g})$ is denoted by $Z(\mathfrak{g})$.

Fix a subset $\Theta\subset \Pi$.
This defines a parabolic subalgebra $\mathfrak{p}_\Theta\supset \mathfrak{p}$.
Let $\mathfrak{p}_\Theta = \mathfrak{m}_\Theta\oplus \mathfrak{a}_\Theta \oplus \mathfrak{n}_\Theta$ be a Langlands decomposition such that $\mathfrak{a}_\Theta\subset \mathfrak{a}$ and $\mathfrak{n}_\Theta\supset \mathfrak{n}$.
Denote the corresponding parabolic subgroup by $P_{\Theta,\R} = M_{\Theta,\R} A_{\Theta,\R} N_{\Theta,\R}\subset G_\R$.
Let $\overline{\mathfrak{p}}_\Theta = \mathfrak{m}_\Theta \oplus \mathfrak{a}_\Theta\oplus\overline{\mathfrak{n}}_\Theta$ be a Langlands decomposition of the opposite subalgebra of $\mathfrak{p}_\Theta$.
Put $\mathfrak{l}_\Theta = \mathfrak{m}_\Theta\oplus \mathfrak{a}_\Theta$, $L_{\Theta,\R} = M_{\Theta,\R}A_{\Theta,\R}$ and $K_\Theta = M_\Theta\cap K$.

Let $\HC_\Theta$ be the category of finitely generated $Z(\mathfrak{g})$-finite $(\mathfrak{g},K_\Theta N_\Theta)$-modules.
Here, a $(\mathfrak{g},K_\Theta N_\Theta)$-module is a vector space equipped with a $\mathfrak{g}$-action and
an algebraic $K_\Theta N_\Theta$-action satisfying the obvious compatibility.
(We do not assume that the action of $K_\Theta N_\Theta$ is semisimple.)
When we emphasis the group $G_\R$, we denote it by $\HC_\Theta^{G_\R}$.
If $\Theta = \Pi$, then $\HC_\Pi$ is the category of Harish-Chandra modules (namely, finite length $(\mathfrak{g},K)$-modules) of $G_\R$~\cite[3.4.7.~Corollary, 4.2.1.~Theorem]{MR929683}.
We will prove that $\HC_\Theta$ is equal to the category $\mathcal{O}'_{P_{\Theta,\R}}$ defined by Hecht-Schmid~\cite{MR705884} (Lemma~\ref{lem:HC,a_Theta decomp} and \ref{lem:n-homology preserve HC}).

For $\mu\in \mathfrak{a}_\Theta^*$ and an $\mathfrak{a}_\Theta$-module $V$, denote the generalized $\mu$-weight space by $\wtsp_{\mu}(V)$.
Set $\wt_{\mathfrak{a}_\Theta}(V) = \{\mu\in\mathfrak{a}^*_\Theta\mid \wtsp_{\mu}(V)\ne 0\}$.

\begin{lem}\label{lem:HC,a_Theta decomp}
For a $\mathfrak{g}$-module $V$ and $k\in\Z_{>0}$, set $V_k = \{v\in V\mid \mathfrak{n}_\Theta^k v = 0\}$.
\begin{enumerate}
\item If $V_1$ is $Z(\mathfrak{l}_\Theta)$-finite, then $V_k$ is $Z(\mathfrak{l}_\Theta)$-finite for all $k\in\Z_{>0}$.
\item If $V$ is $Z(\mathfrak{g})$-finite, then $V_1$ is $Z(\mathfrak{l}_\Theta)$-finite.
\item If $V$ is a $Z(\mathfrak{g})$-finite $(\mathfrak{g},N_\Theta)$-module, then $V = \bigoplus_{\mu\in \mathfrak{a}_\Theta^*}\wtsp_\mu(V)$.
\end{enumerate}
\end{lem}
\begin{proof}
(1)
We prove (1) by induction on $k$.
The homomorphism $V\to \Hom_\C(\mathfrak{n}_\Theta^k,V)$ defined by $v\mapsto (u\mapsto uv)$ gives a homomorphism $V_{k + 1}/V_k \to \Hom_\C(\mathfrak{n}_\Theta^k,V_1)$, which is injective.
By Kostant's theorem, $\Hom_\C(\mathfrak{n}_\Theta^k,V_1)\simeq (\mathfrak{n}_\Theta^k)^*\otimes V_1$ is $Z(\mathfrak{l}_\Theta)$-finite.
Hence $V_{k + 1}$ is $Z(\mathfrak{l}_\Theta)$-finite by inductive hypothesis.

(2)
Consider the following homomorphism:
$\varphi\colon Z(\mathfrak{g})\hookrightarrow U(\mathfrak{g}) = U(\mathfrak{l}_\Theta)\oplus (\overline{\mathfrak{n}}_\Theta U(\mathfrak{g}) + U(\mathfrak{g})\mathfrak{n}_\Theta)\twoheadrightarrow U(\mathfrak{l}_\Theta)$.
Then by a theorem of Harish-Chandra, the image of this homomorphism is contained in $Z(\mathfrak{l}_\Theta)$ and $Z(\mathfrak{l}_\Theta)$ is a finite $Z(\mathfrak{g})$-algebra.
Moreover, the action of $Z(\mathfrak{g})$ on $V_1$ factors through $\varphi$.
Set $I = \Ann_{Z(\mathfrak{g})}V$.
Then $\varphi(I)Z(\mathfrak{l}_\Theta)V_1 = 0$ and 
$\varphi(I)Z(\mathfrak{l}_\Theta)\subset Z(\mathfrak{l}_\Theta)$ has a finite codimension.

(3)
By (1) and (2), $V_k$ is $Z(\mathfrak{l}_\Theta)$-finite.
Since $\mathfrak{a}_\Theta$ is contained in the center of $\mathfrak{l}_\Theta$, we have $U(\mathfrak{a}_\Theta)\subset Z(\mathfrak{l}_\Theta)$.
Therefore, $V_k$ is $U(\mathfrak{a}_\Theta)$-finite.
Hence we have $V_k = \bigoplus_{\mu\in\mathfrak{a}_\Theta^*}\wtsp_{\mu}(V_k)$.
By the assumption, we have $V = \bigcup_k V_k$.
We get the lemma.
\end{proof}

In particular, if $V\in \HC_\Theta$ then $V = \bigoplus_{\mu\in\mathfrak{a}_\Theta^*}\wtsp_\mu(V)$.
Let $\Theta_1,\Theta_2\subset\Pi$ such that $\Theta_2\subset \Theta\subset\Theta_1$.
For such data, we define $J_{\Theta_2,\Theta_1}(V)$ and $\widehat{J}_{\Theta_2,\Theta_1}(V)$ as follows:
\begin{align*}
	\widehat{J}_{\Theta_2,\Theta_1}(V) & = \varprojlim_k V/(\mathfrak{m}_{\Theta_1}\cap\overline{\mathfrak{n}}_{\Theta_2})^kV,\\
	J_{\Theta_2,\Theta_1}(V) & = \{v\in \widehat{J}_{\Theta_2,\Theta_1}(V)\mid \text{$\mathfrak{n}_{\Theta_2}^kv = 0$ for some $k$}\}.
\end{align*}
If $\Theta_1 = \Pi$ and $\Theta_2 = \emptyset$, then $J_{\Theta_2,\Theta_1}(V)$ is called the \emph{Jacquet module} of $V$.
If $\Theta_1 = \Theta_2 = \Theta$, $J_{\Theta_2,\Theta_1}(V) = V$.
We will prove the following properties in this section:
\begin{itemize}
\item For $V\in\HC_{\Theta}$, $J_{\Theta_2,\Theta_1}(V)$ is independent of $\Theta_1$ and an object of $\HC_{\Theta_2}$.
\item For $V\in\HC_{\Theta_1}$ and $\Theta_3\subset\Theta_2\subset\Theta_1\subset\Pi$, we have $J_{\Theta_3,\Theta_2}\circ J_{\Theta_2,\Theta_1}(V)\simeq J_{\Theta_3,\Theta_1}(V)$.
\item The functor $J_{\Theta_2,\Theta_1}\colon\HC_{\Theta_1}\to\HC_{\Theta_2}$ is exact.
\end{itemize}
So the usual Jacquet functor is decomposed into the composite of $J_{\Theta_2,\Theta_1}$'s.

\begin{lem}\label{lem:n-homology preserve HC}
\begin{enumerate}
\item If $V\in\HC_\Theta^{G_\R}$ then $\wtsp_{\mu_1}(V)\in\HC_\Theta^{L_{\Theta_1,\R}}$ for $\mu_1\in\mathfrak{a}_{\Theta_1}^*$.
\item If $V\in\HC_\Theta^{G_\R}$ then $V/\overline{\mathfrak{n}}_{\Theta}V\in \HC_{\Theta}^{L_{\Theta,\R}}$.
\item If $V\in\HC_\Pi^{G_\R}$ then $V/\overline{\mathfrak{n}}_{\Theta}V\in \HC_{\Theta}^{L_{\Theta,\R}}$.
\item For $V\in\HC_\Theta$ and $\mu_1\in\mathfrak{a}_{\Theta_1}^*$, $\wtsp_{\mu_1}(V/(\mathfrak{m}_{\Theta_1}\cap \overline{\mathfrak{n}}_{\Theta_2})^kV)$ is a Harish-Chandra module of $L_{\Theta_2,\R}$.
\end{enumerate}
\end{lem}
\begin{proof}
(1)
It is easy to see that $\wtsp_{\mu_1}(V)$ is a $(\mathfrak{g},K_{\Theta}(M_{\Theta_1}\cap N_{\Theta}))$-module.
It is sufficient to prove that $\wtsp_{\mu_1}(V)$ is a finitely generated $U(\mathfrak{l}_{\Theta_1})$-module and $Z(\mathfrak{l}_{\Theta_1})$-finite.

Since $V\in \HC_\Theta$, $V$ is generated by a finite-dimensional subspace $W$ of $V$.
By Lemma~\ref{lem:HC,a_Theta decomp} (3), the action of $\mathfrak{a}_\Theta$ on $V$ is locally finite.
By the definition of $\HC_\Theta$, the action of $\mathfrak{n}_\Theta$ on $V$ is locally finite.
Hence we may assume that $W$ is $\mathfrak{n}_{\Theta}\oplus\mathfrak{a}_{\Theta}$-stable.
In particular, $W$ is $\mathfrak{n}_{\Theta_1}\oplus\mathfrak{a}_{\Theta_1}$-stable.
Therefore, we have $V = U(\mathfrak{m}_{\Theta_1})U(\overline{\mathfrak{n}}_{\Theta_1})W$.
From this, we get 
\[
	\wtsp_{\mu_1}(V) = U(\mathfrak{m}_{\Theta_1})\left(\sum_{\mu_1'\in\wt_{\mathfrak{a}_{\Theta_1}}(W)}\wtsp_{\mu_1 - \mu_1'}(U(\overline{\mathfrak{n}}_{\Theta_1}))\wtsp_{\mu_1'}(W)\right).
\]
Since $W$ is finite-dimensional, $\wt_{\mathfrak{a}_{\Theta_1}}(W)$ is finite.
For each $\mu_1'$, $\wtsp_{\mu_1 - \mu_1'}(U(\overline{\mathfrak{n}}_{\Theta_1}))$ is finite-dimensional.
Therefore, $\sum_{\mu_1'\in\wt_{\mathfrak{a}_{\Theta_1}}(W)}\wtsp_{\mu_1 - \mu_1'}(U(\overline{\mathfrak{n}}_{\Theta_1}))\wtsp_{\mu_1'}(W)$ is finite-dimensional.
Hence $\wtsp_{\mu_1}(V)$ is a finitely generated $U(\mathfrak{m}_{\Theta_1})$-module.

Put $V_k = \{v\in V\mid \mathfrak{n}_{\Theta_1}^kv = 0\}$.
Since $\wtsp_{\mu_1}(V)$ is finitely generated $U(\mathfrak{m}_{\Theta_1})$-module, we can take $\mathfrak{n}_{\Theta_1}$-stable subspace $W'$ such that $\wtsp_{\mu_1}(V)\subset U(\mathfrak{m}_{\Theta_1})W'$.
Since $W'$ is finite-dimensional, $W'\subset V_k$ for some $k\in\Z_{\ge 0}$.
By Lemma~\ref{lem:HC,a_Theta decomp}, $V_k$ is $Z(\mathfrak{l}_{\Theta_1})$-finite.
Hence $\wtsp_{\mu_1}(V)$ is $Z(\mathfrak{l}_{\Theta_1})$-finite.

(2)
As above, we can take a finite-dimensional $\mathfrak{n}_{\Theta}\oplus\mathfrak{a}_{\Theta}$-stable submodule $W$ which generates $V$ as a $\mathfrak{g}$-module.
Then we have $V = U(\overline{\mathfrak{n}}_{\Theta})U(\mathfrak{m}_{\Theta})W$.
Hence we get a surjective homomorphism $U(\mathfrak{m}_{\Theta})W\to V/\overline{\mathfrak{n}}_{\Theta}V$.
It is sufficient to prove that $U(\mathfrak{m}_{\Theta})W\in \HC_{\Theta}^{L_{\Theta,\R}}$.
Since $W$ is finite-dimensional, $W\subset\bigoplus_{\mu\in\Lambda}\wtsp_{\mu}(V)$ for a finite subset $\Lambda\subset\mathfrak{a}_{\Theta}^*$.
Each $\wtsp_\mu(V)$ is $\mathfrak{m}_\Theta$-stable.
Therefore, $U(\mathfrak{m}_\Theta)W\subset\bigoplus_{\mu\in\Lambda}\wtsp_{\mu}(V)$.
By (1), $\wtsp_{\mu}(V)\in\HC_{\Theta}^{L_{\Theta,\R}}$.
Hence we have $U(\mathfrak{m}_{\Theta})W\in\HC_{\Theta}^{L_{\Theta,\R}}$.

(3)
Recall that an object of $\HC_{\Pi}^{G_\R}$ is a Harish-Chandra module of $G_\R$.
Hence this is \cite[Proposition~2.24]{MR716371}.

(4)
By Lemma~\ref{lem:HC,a_Theta decomp}, we have $V = \bigoplus_{\mu_1\in\mathfrak{a}_{\Theta_1}^*}\wtsp_{\mu_1}(V)$.
Since $\mathfrak{m}_{\Theta_1}$ has an $\mathfrak{a}_{\Theta_1}$-weight $0$, the action of $\mathfrak{m}_\Theta$ preserves each $\Gamma_{\mu_1}(V)$.
Therefore, we have $\wtsp_{\mu_1}(V/(\mathfrak{m}_{\Theta_1}\cap \overline{\mathfrak{n}}_{\Theta_2})^kV) = \wtsp_{\mu_1}(V)/(\mathfrak{m}_{\Theta_1}\cap \overline{\mathfrak{n}}_{\Theta_2})^k\wtsp_{\mu_1}(V)$.
By (1), $\wtsp_{\mu_1}(V)\in\HC_{\Theta}^{L_{\Theta_1,\R}}$.
Hence we may assume that $\Theta_1 = \Pi$.
So we have $\mathfrak{m}_{\Theta_1}\cap\overline{\mathfrak{n}}_{\Theta_2} = \overline{\mathfrak{n}}_{\Theta_2}$.

By (2), $V' = V/\overline{\mathfrak{n}}_{\Theta}V\in \HC_{\Theta}^{L_{\Theta,\R}}$.
By (3), $V'' = V'/(\mathfrak{m}_{\Theta}\cap\overline{\mathfrak{n}}_{\Theta_2})V'\in\HC_{\Theta_2}^{L_{\Theta_2,\R}}$.
Since $V'' = V/\overline{\mathfrak{n}}_{\Theta_2}V$, we get the lemma for $k = 1$.

For a general $k$, we can prove the lemma by induction on $k$ using an exact sequence $\overline{\mathfrak{n}}_{\Theta_2}^k\otimes(V/\overline{\mathfrak{n}}_{\Theta_2}V)\to V/\overline{\mathfrak{n}}_{\Theta_2}^{k + 1}V\to V/\overline{\mathfrak{n}}_{\Theta_2}^{k}V\to 0$.
\end{proof}

We now prove that the length of an object of $\HC_\Theta$ is finite.
\begin{lem}\label{lem:suff cond for fin len}
Assume that a $(\mathfrak{g},N_\Theta)$-module $V$ satisfies the following conditions:
\begin{enumerate}
\item The module $V$ is $Z(\mathfrak{g})$-finite.
\item For all $\mu\in\mathfrak{a}_\Theta^*$, $\wtsp_\mu(V)$ has a finite length as a $\mathfrak{m}_\Theta$-module.
\end{enumerate}
Then $V$ has a finite length.
\end{lem}
\begin{proof}
Let $\varphi\colon Z(\mathfrak{g})\to Z(\mathfrak{l}_\Theta)$ be a homomorphism defined in the proof of Lemma~\ref{lem:HC,a_Theta decomp}.
Set $J = \varphi(\Ann_{Z(\mathfrak{g})}V)Z(\mathfrak{l}_\Theta)$.
Then $J$ has a finite codimension.
In particular, there exists a finite subset $\Lambda\subset\mathfrak{a}_\Theta^*$ which gives all maximal ideals of $U(\mathfrak{a}_\Theta)/(J\cap U(\mathfrak{a}_\Theta))$.
Let $V'$ be a subquotient of $V$.
Then we have $J(V')^{\mathfrak{n}_\Theta} = 0$, hence $(V')^{\mathfrak{n}_\Theta} \subset \bigoplus_{\mu\in\Lambda}\Gamma_{\mu}(V')$.
Since $V'$ is a $(\mathfrak{g},N_\Theta)$-module, the space $(V')^{\mathfrak{n}_\Theta}$ is non-zero.
Hence the length of the $\mathfrak{g}$-module $V$ is less than or equal to the sum of the length of $\mathfrak{m}_\Theta$-modules $\wtsp_\mu(V)$ for $\mu\in\Lambda$.
It is finite by the assumption.
\end{proof}

\begin{cor}
Each object in $\HC_\Theta$ has a finite length.
\end{cor}

\begin{proof}
The condition of (1) in the previous lemma is satisfied by the definition of $\HC_\Theta$.
For $V\in \HC_\Theta$ and $\mu\in \mathfrak{a}_\Theta^*$, $\Gamma_{\mu}(V)$ is a Harish-Chandra module of $L_{\Theta,\R}$ by Lemma~\ref{lem:n-homology preserve HC} (4). (Take $\Theta_1 = \Theta_2 = \Theta$.)
Hence it has a finite length.
By the previous lemma, we get the corollary.
\end{proof}

For a subset $\Lambda\subset\mathfrak{a}_{\Theta}^*$, put $\Lambda - \Z_{\ge 0}\Pi|_{\mathfrak{a}_{\Theta}} = \{\mu - \sum_{\alpha\in\Pi}n_\alpha \alpha|_{\mathfrak{a}_{\Theta}}\mid \mu\in \Lambda,\ n_\alpha\in\Z_{\ge 0}\}$.

\begin{lem}\label{lem:wt of Jhat}
For $V\in\HC_{\Theta}$, there exists a finite subset $\Lambda_2$ of $\mathfrak{a}_{\Theta_2}^*$ such that $\wt_{\mathfrak{a}_{\Theta_2}}(\widehat{J}_{\Theta_2,\Theta_1}(V))\subset\Lambda_2 - \Z_{\ge 0}\Pi|_{\mathfrak{a}_{\Theta_2}}$.
\end{lem}
\begin{proof}
Put $\mathfrak{c} = \mathfrak{m}_{\Theta_1}\cap\overline{\mathfrak{n}}_{\Theta_2}$.
As in the proof of Lemma~\ref{lem:n-homology preserve HC}, we can take a finite-dimensional $\mathfrak{n}_{\Theta_1}\oplus\mathfrak{a}_{\Theta_1}$-stable subspace $W$ such that $V = U(\mathfrak{g})W$.
Put $\Lambda_1 = \wt_{\mathfrak{a}_{\Theta_1}}(W)$.
Then this is finite and, since $V = U(\mathfrak{m}_\Theta)U(\overline{\mathfrak{n}}_\Theta)W$, $\wt_{\mathfrak{a}_{\Theta_1}}(V)\subset \Lambda_1-\Z_{\ge 0}\Pi|_{\mathfrak{a}_{\Theta_1}}$.
Set $\Lambda_2 = \bigcup_{\mu_1\in\Lambda_1}\wt_{\mathfrak{a}_{\Theta_2}}(\wtsp_{\mu_1}(V/\mathfrak{c}V)) = \{\mu_2\in \wt_{\mathfrak{a}_{\Theta_2}}(V/\mathfrak{c}V)\mid \mu_2|_{\mathfrak{a}_{\Theta_1}}\in\Lambda_1\}$.
By Lemma~\ref{lem:n-homology preserve HC} (4), $\Gamma_{\mu_1}(V/\mathfrak{c}V)$ is a Harish-Chandra module of $L_{\Theta_2,\R}$.
In particular, it is $Z(\mathfrak{l}_\Theta)$-finite.
Therefore, it is $U(\mathfrak{a}_\Theta)$-finite since $\mathfrak{a}_\Theta$ is a subalgebra of the center of $\mathfrak{l}_\Theta$.
Hence $\wt_{\mathfrak{a}_{\Theta_2}}(\Gamma_{\mu_1}(V/\mathfrak{c}V))$ is finite.
This implies that $\Lambda_2$ is finite.
We also have $\wt_{\mathfrak{a}_{\Theta_2}}(V/\mathfrak{c}V)\subset\Lambda_2 - \Z_{\ge 0}\Pi|_{\mathfrak{a}_{\Theta_2}}$.

We prove $\wt_{\mathfrak{a}_{\Theta_2}}(V/\mathfrak{c}^kV)\subset \Lambda_2 - \Z_{\ge 0}\Pi|_{\mathfrak{a}_{\Theta_2}}$ by induction on $k$.
Then we get the lemma.
From an exact sequence $\mathfrak{c}\otimes (V/\mathfrak{c}^kV)\to V/\mathfrak{c}^{k + 1}V\to V/\mathfrak{c}V\to 0$, we have
\[
	\wt(V/\mathfrak{c}^{k + 1}V)\subset \wt(\mathfrak{c}\otimes (V/\mathfrak{c}^kV))\cup \wt(V/\mathfrak{c}V)\subset \Lambda_2 - \Z_{\ge 0}\Pi|_{\mathfrak{a}_{\Theta_2}}
\]
by the inductive hypothesis.
\end{proof}

\begin{lem}\label{lem:decomp int we sps of J}
For $V\in \HC_\Theta$, we have $J_{\Theta_2,\Theta_1}(V) = \bigoplus_{\mu_2\in \mathfrak{a}_{\Theta_2}^*}\wtsp_{\mu_2}(\widehat{J}_{\Theta_2,\Theta_1}(V))$.
Therefore, $\Gamma_{\mu_2}(\widehat{J}_{\Theta_2,\Theta_1}(V)) = \Gamma_{\mu_2}(J_{\Theta_2,\Theta_1}(V))$.
\end{lem}
\begin{proof}
By Lemma~\ref{lem:wt of Jhat}, the right hand side is contained in the left hand side.
On the other hand, by Lemma~\ref{lem:HC,a_Theta decomp} (3), we have $J_{\Theta_2,\Theta_1}(V) = \bigoplus_{\mu_2\in\mathfrak{a}_{\Theta_2}^*}\wtsp_{\mu_2}(J_{\Theta_2,\Theta_1}(V))$.
This is a subspace of the right hand side.
\end{proof}

Notice that $V/(\mathfrak{m}_{\Theta_1}\cap \overline{\mathfrak{n}}_{\Theta_2})^k V\simeq \widehat{J}_{\Theta_2,\Theta_1}(V)/(\mathfrak{m}_{\Theta_1}\cap \overline{\mathfrak{n}}_{\Theta_2})^k \widehat{J}_{\Theta_2,\Theta_1}(V)$ by the definition.

\begin{lem}\label{lem:wt sp is stable}
For $\mu_2\in \mathfrak{a}_{\Theta_2}^*$, there exists $k\in\Z_{\ge 0}$ such that $\wtsp_{\mu_2}(\widehat{J}_{\Theta_2,\Theta_1}(V))\to \wtsp_{\mu_2}(V/(\mathfrak{m}_{\Theta_1}\cap \overline{\mathfrak{n}}_{\Theta_2})^k V)$ is an isomorphism.
\end{lem}
\begin{proof}
Put $\mathfrak{c} = \mathfrak{m}_{\Theta_1}\cap \overline{\mathfrak{n}}_{\Theta_2}$.
Take $\Lambda_2$ as in Lemma~\ref{lem:wt of Jhat}.
Let $k\in\Z_{\ge 0}$ such that for any $\alpha_1,\dots,\alpha_k\in\Sigma^+$ we have $\mu_2 \not\in (\Lambda_2 - \Z_{\ge 0}\Pi|_{\mathfrak{a}_{\Theta_2}}) - (\alpha_1 + \dotsb + \alpha_k)$.
Then we have $\wtsp_{\mu_2}(\mathfrak{c}^k \widehat{J}_{\Theta_2,\Theta_1}(V)) = 0$.
By the exact sequence $0\to \mathfrak{c}^k \widehat{J}_{\Theta_2,\Theta_1}(V)\to \widehat{J}_{\Theta_2,\Theta_1}(V)\to V/\mathfrak{c}^k V\to 0$, we have $0 = \wtsp_{\mu_2}(\mathfrak{c}^k \widehat{J}_{\Theta_2,\Theta_1}(V)) \to \wtsp_{\mu_2}(\widehat{J}_{\Theta_2,\Theta_1}(V))\to \wtsp_{\mu_2}(V/\mathfrak{c}^kV)\to 0$.
Hence we get the lemma.
\end{proof}

\begin{lem}
For $V\in\HC_\Theta$ and $\mu_2\in\mathfrak{a}_{\Theta_2}^*$, $\wtsp_{\mu_2}(J_{\Theta_2,\Theta_1}(V))$ is a Harish-Chandra module for $L_{\Theta_2,\R}$.
\end{lem}
\begin{proof}
This follows from Lemma~\ref{lem:n-homology preserve HC}, Lemma~\ref{lem:decomp int we sps of J} and Lemma~\ref{lem:wt sp is stable}.
\end{proof}

If $\Theta = \Theta_1 = \Pi$, the following proposition is \cite[(34) Lemma]{MR705884}.

\begin{prop}\label{prop:image of Jacquet functor}
If $V\in \HC_\Theta$, then $J_{\Theta_2,\Theta_1}(V)\in \HC_{\Theta_2}$.
\end{prop}
\begin{proof}
It is sufficient to prove that $J_{\Theta_2,\Theta_1}(V)$ has a finite length.
This follows from Lemma~\ref{lem:suff cond for fin len} and the previous lemma.
\end{proof}

Hence $J_{\Theta_2,\Theta_1}$ defines a functor $\HC_{\Theta}\to \HC_{\Theta_2}$.

\begin{prop}\label{prop:chain rule of J}
For $\Theta_3\subset \Theta_2\subset \Theta\subset\Theta_1\subset\Pi$, we have $J_{\Theta_3,\Theta_2}\circ J_{\Theta_2,\Theta_1}\simeq J_{\Theta_3,\Theta_1}\colon \HC_{\Theta}\to \HC_{\Theta_3}$.
\end{prop}

We use the following lemma.
\begin{lem}
Let $\mathfrak{c}$ be a finite-dimensional nilpotent Lie algebra,
$\mathfrak{c}_1,\mathfrak{c}_2\subset \mathfrak{c}$ Lie subalgebras such that:
\begin{itemize}
\item $\mathfrak{c}_2$ is an ideal of $\mathfrak{c}$.
\item $\mathfrak{c} = \mathfrak{c}_1\oplus \mathfrak{c}_2$.
\end{itemize}
Then for all $k_1,k_2\in\Z_{\ge 0}$ there exists $n\in\Z_{\ge 0}$ such that $\mathfrak{c}^n\subset \mathfrak{c}_1^{k_1}U(\mathfrak{c}) + \mathfrak{c}_2^{k_2}U(\mathfrak{c})$.
\end{lem}
\begin{proof}
Set $V = U(\mathfrak{c})/(\mathfrak{c}_1^{k_1}U(\mathfrak{c}) + \mathfrak{c}_2^{k_2}U(\mathfrak{c}))$, $v_0 = 1\in V$.
Then $V$ is a right $U(\mathfrak{c})$-module, $V = v_0U(\mathfrak{c})$ and $v_0\mathfrak{c}_1^{k_1} = v_0\mathfrak{c}_2^{k_2} = 0$.
We have $V = v_0U(\mathfrak{c}_1)U(\mathfrak{c}_2)$.
Since $v_0\mathfrak{c}_1^{k_1} = 0$, $v_0U(\mathfrak{c}_1)$ is finite-dimensional.
By the assumption, $\mathfrak{c}_2$ is an ideal of $\mathfrak{c}$.
Therefore, $v_0U(\mathfrak{c}_1)\mathfrak{c}_2^{k_2} = v_0\mathfrak{c}_2^{k_2}U(\mathfrak{c}) = 0$.
Hence $V$ is finite-dimensional.
Since a finite-dimensional irreducible representation of $\mathfrak{c}$ is a character, $V$ is given by an extension of characters.
As $v_0\mathfrak{c}_1^{k_1} = v_0\mathfrak{c}_2^{k_2} = 0$ and $\mathfrak{c}$ is nilpotent,
for every $v\in V$ there exist integers $l_1$, $l_2$
such that $v\mathfrak{c}_1^{l_1} = v\mathfrak{c}_2^{l_2} = 0$.
This implies that each irreducible subquotient of $V$ is the trivial representation.
Hence there exists $n$ such that $v_0\mathfrak{c}^n = 0$.
This completes the proof.
\end{proof}

\begin{proof}[Proof of Proposition~\ref{prop:chain rule of J}]
We prove that for each $\mu_3\in\mathfrak{a}_{\Theta_3}^*$ the generalized $\mu_3$-weight spaces of both sides are isomorphic.
Put $\mathfrak{c}_1 = \mathfrak{m}_{\Theta_2}\cap \overline{\mathfrak{n}}_{\Theta_3}$, $\mathfrak{c}_2 = \mathfrak{m}_{\Theta_1}\cap \overline{\mathfrak{n}}_{\Theta_2}$ and $\mathfrak{c} = \mathfrak{c}_1\oplus \mathfrak{c}_2$.
Then $\mathfrak{c}_1,\mathfrak{c}_2$ satisfies the assumption of the previous lemma and $\mathfrak{c} = \mathfrak{m}_{\Theta_1}\cap\overline{\mathfrak{n}}_{\Theta_3}$.
Put $\mu_2 = \mu_3|_{\mathfrak{a}_{\Theta_2}}$.
By Lemma~\ref{lem:decomp int we sps of J} and Lemma~\ref{lem:wt sp is stable}, for sufficiently large $k_1,k_2$, we have
\begin{align*}
\Gamma_{\mu_3}(J_{\Theta_3,\Theta_2}(J_{\Theta_2,\Theta_1}(V))) & = \Gamma_{\mu_3}(J_{\Theta_2,\Theta_1}(V)/\mathfrak{c}_1^{k_1}J_{\Theta_2,\Theta_1}(V))\\
& = \Gamma_{\mu_3}\Gamma_{\mu_2}(J_{\Theta_2,\Theta_1}(V)/\mathfrak{c}_1^{k_1}J_{\Theta_2,\Theta_1}(V))\\
& = \Gamma_{\mu_3}(\Gamma_{\mu_2}(J_{\Theta_2,\Theta_1}(V))/\mathfrak{c}_1^{k_1}\Gamma_{\mu_2}(J_{\Theta_2,\Theta_1}(V)))\\
& = \Gamma_{\mu_3}(\Gamma_{\mu_2}(V/\mathfrak{c}_2^{k_2}V)/\mathfrak{c}_1^{k_1}\Gamma_{\mu_2}(V/\mathfrak{c}_2^{k_2}V))\\
& = \Gamma_{\mu_3}(\Gamma_{\mu_2}((V/\mathfrak{c}_2^{k_2}V)/\mathfrak{c}_1^{k_1}(V/\mathfrak{c}_2^{k_2}V)))\\
& = \Gamma_{\mu_3}(\Gamma_{\mu_2}(V/(\mathfrak{c}_1^{k_1}V + \mathfrak{c}_2^{k_2}V)))\\
& = \Gamma_{\mu_3}(V/(\mathfrak{c}_1^{k_1}V + \mathfrak{c}_2^{k_2}V)).
\end{align*}
We also have
\[
	\wtsp_{\mu_3}(J_{\Theta_3,\Theta_1}(V)) \simeq \wtsp_{\mu_3}(V/\mathfrak{c}^{k}V)
\]
for sufficiently large $k$.
Fix $k_1,k_2$ and take $n$ as in the previous lemma.
We may assume $k\le k_1,k_2\le n$.
Consider the following homomorphism:
\[
	\wtsp_{\mu_3}(V/\mathfrak{c}^{n}V)
	\to
	\wtsp_{\mu_3}(V/(\mathfrak{c}_1^{k_1}V + \mathfrak{c}_2^{k_2}V))
	\to
	\wtsp_{\mu_3}(V/\mathfrak{c}^kV).
\]
Since $V/\mathfrak{c}^nV$ is decomposed into the generalized $\mathfrak{a}_{\Theta_3}$-weight spaces (this follows from the fact that $V/\mathfrak{c}^nV$ is a Harish-Chandra module of $L_{\Theta_3,\R}$), the first homomorphism is surjective.
If $k_1,k_2,n,k$ is sufficiently large, the composition of this homomorphism is isomorphic by Lemma~\ref{lem:wt sp is stable}.
Hence the first homomorphism is injective.
We get the proposition.
\end{proof}

\begin{prop}
The functor $J_{\Theta_2,\Theta_1}\colon \HC_{\Theta}\to \HC_{\Theta_2}$ is independent of $\Theta_1$.
\end{prop}

\begin{proof}
By the previous proposition, $J_{\Theta_2,\Theta_1} = J_{\Theta_2,\Theta}\circ J_{\Theta,\Theta_1}$.
Hence it is sufficient to prove that $J_{\Theta,\Theta_1}(V) \simeq V$ for $V\in\HC_{\Theta}$.
We compare the $\mu$-weight spaces for each $\mu\in\mathfrak{a}_\Theta^*$.

Put $\mathfrak{c} = \mathfrak{m}_{\Theta_1}\cap\overline{\mathfrak{n}}_{\Theta}$.
Take a finite subset $\Lambda\subset\mathfrak{a}_{\Theta}^*$ such that $\wt_{\mathfrak{a}_{\Theta}}(V)\subset \Lambda - \Z_{\ge 0}\Pi|_{\mathfrak{a}_{\Theta}}$.
Then for a sufficiently large $k$, for any $\alpha_1,\dots,\alpha_k\in\Pi$ we have $\mu \not\in\Lambda - \Z_{\ge 0}\Pi|_{\mathfrak{a}_{\Theta}} - (\alpha_1 + \dotsm + \alpha_k)$.
Hence we have $\mu\not\in\wt_{\mathfrak{a}_\Theta}(\mathfrak{c}^kV)$.
This implies $\wtsp_\mu(V)\simeq \wtsp_\mu(V/\mathfrak{c}^kV)$.
On the other hand, the right hand side is isomorphic to $\wtsp_\mu(J_{\Theta,\Theta_1}(V))$ for a sufficiently large $k$ by Lemma~\ref{lem:wt sp is stable}.
We get the proposition.
\end{proof}

\begin{lem}
Each $V\in\HC_\Theta$ is finitely generated as a $U(\overline{\mathfrak{n}})$-module.
\end{lem}
\begin{proof}
Take a finite-dimensional $\mathfrak{n}_\Theta\oplus \mathfrak{a}_\Theta$-stable subspace $W$ of $V$ which generates $V$ as a $\mathfrak{g}$-module.
Then $U(\mathfrak{m}_\Theta)W\subset \bigoplus_{\mu\in\Lambda}\wtsp_\mu(V)$ for some finite subset $\Lambda\subset\mathfrak{a}^*_\Theta$.
Hence $U(\mathfrak{m}_\Theta)W$ is a Harish-Chandra module of $L_{\Theta,\R}$.
Therefore, by a theorem of Casselman-Osborne~\cite[2.3 Theorem]{MR0480884}, $U(\mathfrak{m}_\Theta)W$ is finitely generated as a $(\mathfrak{m}_\Theta\cap \overline{\mathfrak{n}})$-module.
Since $V = U(\overline{\mathfrak{n}}_\Theta)U(\mathfrak{m}_\Theta)W$, $V$ is a finitely generated $U(\overline{\mathfrak{n}}_\Theta)U(\mathfrak{m}_\Theta\cap \overline{\mathfrak{n}}) = U(\overline{\mathfrak{n}})$-module.
\end{proof}

\begin{prop}
The functor $J_{\Theta_2,\Theta_1}\colon \HC_{\Theta_1}\to\HC_{\Theta_2}$ is exact.
\end{prop}
\begin{proof}
If $\Theta_2 = \emptyset$ and $\Theta_1 = \Pi$, this proposition is well-known~\cite[4.1.5. Theorem]{MR929683}.
The key point of the proof is the Artin-Rees property and that $V\in\HC_\Pi$ is finitely generated as a $U(\overline{\mathfrak{n}})$-module.
Hence the usual proof is applicable for our situation using the above lemma.
\end{proof}

\section{The geometric Jacquet functor}\label{sec:The geometric Jacquet functor}
In this section, we recall an argument of Emerton-Nadler-Vilonen~\cite{MR2096674}.
For $\Theta\subset \Pi$, let $\HC_{\Theta,\rho}$ be the category of $V\in\HC_\Theta$ whose infinitesimal character is the same as that of the trivial representation.
Fix a Borel subgroup $B$ of $G$.
Then by the Beilinson-Bernstein correspondence and the Riemann-Hilbert correspondence, we have an equivalence of categories $\Delta\colon\HC_{\Theta,\rho}\simeq \Perv_{K_\Theta N_\Theta}(G/B)$ where $\Perv_{K_\Theta N_\Theta}(G/B)$ is the category of $K_\Theta N_\Theta$-equivariant perverse sheaves on $G/B$.

Fix a cocharacter $\nu\colon \Cbatu\to A$ such that $\langle \nu,\alpha\rangle \ge 0$ for all $\alpha\in\Pi$.
Define $a_\nu\colon G/B\times \Cbatu \to G/B$ by $(gB,t)\mapsto \nu(t)gB$.
Let $R\psi$ be the nearby cycle functor with respect to $G/B\times \Caff^1\to \Caff^1$.
For $\mathscr{F}\in\Perv(G/B)$, put $\Psi_\nu(\mathscr{F}) = R\psi a_\nu^*\mathscr{F}$.
Then the main theorem of \cite{MR2096674} is the following.
\begin{thm}[Emerton-Nadler-Vilonen~{\cite[Theorem~1.1]{MR2096674}}]
Assume that $\nu$ is regular.
We have $\Delta\circ J_{\emptyset,\Pi} \simeq \Psi_\nu\circ \Delta\colon \HC_{\Pi,\rho}\to \Perv(G/B)$.
\end{thm}
Their argument can be applicable for a general $\nu$.
Namely, we can prove the following theorem.
\begin{thm}\label{thm:Emerton-Nadler-Vilonen2}
Set $\Theta = \{\alpha\in\Pi\mid \langle \alpha,\nu\rangle = 0\}$.
Then we have $\Delta\circ J_{\Theta,\Theta'}\simeq \Psi_\nu\circ \Delta\colon\HC_{\Theta',\rho}\to\Perv(G/B)$ for all $\Theta'\subset\Pi$ such that $\Theta\subset\Theta'$.
\end{thm}
We review the proof.
Let $V\in\HC_{\Theta',\rho}$.
First we construct a filtration on $V$.
To construct it, we prove the following lemma.
\begin{lem}
We have the following.
\begin{enumerate}
\item We have $J_{\Theta,\Pi}(V)/\overline{\mathfrak{n}}_{\Theta}^kJ_{\Theta,\Pi}(V)\simeq \widehat{J}_{\Theta,\Pi}(V)/\overline{\mathfrak{n}}_{\Theta}^k\widehat{J}_{\Theta,\Pi}(V)$.
\item We have $\widehat{J}_{\Theta,\Pi}(V)\simeq \widehat{J}_{\Theta,\Pi}(J_{\Theta,\Pi}(V))$.
\item We have $\widehat{J}_{\Theta,\Pi}(V)\simeq\prod_{\mu\in\mathfrak{a}_{\Theta}^*}\wtsp_{\mu}(J_{\Theta,\Pi}(V))$.
\item The homomorphism $V\to \widehat{J}_{\Theta,\Pi}(V)$ is injective.
\end{enumerate}
\end{lem}
\begin{proof}
(1)
Since both sides are decomposed into generalized $\mathfrak{a}_{\Theta}^*$-weight spaces, it is sufficient to prove that the $\mu$-weight spaces of both sides are isomorphic for every $\mu\in\mathfrak{a}_{\Theta}^*$.
We have
\[
	\wtsp_{\mu}(J_{\Theta,\Pi}(V)/\overline{\mathfrak{n}}_{\Theta}^kJ_{\Theta,\Pi}(V))
	\simeq
	\wtsp_{\mu}(J_{\Theta,\Pi}(V))/\wtsp_{\mu}(\overline{\mathfrak{n}}_{\Theta}^kJ_{\Theta,\Pi}(V)).
\]
By Lemma~\ref{lem:decomp int we sps of J}, we have $\wtsp_{\mu}(J_{\Theta,\Pi}(V)) = \wtsp_{\mu}(\widehat{J}_{\Theta,\Pi}(V))$.
We also have
\begin{multline*}
\wtsp_{\mu}(\overline{\mathfrak{n}}_{\Theta}^kJ_{\Theta,\Pi}(V))
=
\sum_{\mu' + \mu'' = \mu}\wtsp_{\mu'}(\overline{\mathfrak{n}}_{\Theta}^k)\wtsp_{\mu''}(J_{\Theta,\Pi}(V))\\
=
\sum_{\mu' + \mu'' = \mu}\wtsp_{\mu'}(\overline{\mathfrak{n}}_{\Theta}^k)\wtsp_{\mu''}(\widehat{J}_{\Theta,\Pi}(V))
=
\wtsp_{\mu}(\overline{\mathfrak{n}}_{\Theta}^k\widehat{J}_{\Theta,\Pi}(V)).
\end{multline*}
We get (1).

(2)
This follows from (1).

(3)
By (2), it is sufficient to prove that for $V\in\HC_\Theta$, $\widehat{J}_{\Theta,\Pi}(V)\simeq \prod_{\mu\in\mathfrak{a}_\Theta^*}\wtsp_\mu(V)$.
We can use the proof of \cite[Lemma~2.2]{MR597811}.

(4)
 The kernel of $V\to\widehat{J}_{\Theta,\Pi}(V)$ satisfies
 $\Ker/(\mathfrak{m}_\Theta\cap \overline{\mathfrak{n}}_\Theta)\Ker=0$.
 Therefore, it is sufficient to prove that for $V\in \HC_{\Theta'}$ with $V\neq 0$, $V/(\mathfrak{m}_\Theta\cap \overline{\mathfrak{n}}_\Theta)V\ne 0$. 
We prove $V/\overline{\mathfrak{n}}V\ne 0$.
Put $V' = V/\mathfrak{n}_{\Theta'}V$.
Take a maximal $\mathfrak{a}_{\Theta'}$-weight $\mu'$ of $V$.
Then $\wtsp_{\mu'}(\mathfrak{n}_{\Theta'}V) = 0$.
Hence $\wtsp_{\mu'}(V') = \wtsp_{\mu'}(V)$.
In particular, $V'\ne 0$.
By Lemma~\ref{lem:n-homology preserve HC}, $V'$ is a Harish-Chandra module of $L_{\Theta',\R}$.
By Casselman's subrepresentation theorem, we have $V/\overline{\mathfrak{n}}V = V'/(\mathfrak{m}_{\Theta'}\cap \overline{\mathfrak{n}})V'\ne 0$.
\end{proof}
Using (4), we regard $V$ as a submodule of $\widehat{J}_{\Theta,\Pi}(V)$.
Let $d\nu\colon \C\to \mathfrak{a}$ be the differential of $\nu$ and put $H = d\nu(1)$.
This is an integral dominant element of $\mathfrak{a}_\Theta$.
In general, for an $\mathfrak{a}$-module $V$, let $\wtsp_{H,a}(V)$ be the generalized $H$-eigenspace of $V$ with an eigenvalue $a$.
For $a,a'\in\C$, we define $a\ge_\Z a'$ by $a - a'\in\Z_{\ge 0}$.
For $V\in \HC_{\Theta',\rho}$ and $a\in\C$, define $F_a(\widehat{J}_{\Theta,\Pi}(V))\subset \widehat{J}_{\Theta,\Pi}(V)$ and $F_aV\subset V$ by
\[
	F_a(\widehat{J}_{\Theta,\Pi}(V))=\prod_{a'\le_\Z a}\wtsp_{H,a'}(\widehat{J}_{\Theta,\Pi}(V)),
	\qquad
	F_aV = V\cap F_a(\widehat{J}_{\Theta,\Pi}(V)).
\]
Let $\mathscr{D}$ be the ring of differential operators on $G/B$.
Set $\mathscr{V} = \mathscr{D}\otimes_{U(\mathfrak{g})}V$, $\widetilde{\mathscr{V}} = a_\nu^*\mathscr{V}$ and $\widetilde{V} = \Gamma(G/B\times \Cbatu,\widetilde{\mathscr{V}}) = \C[t,t^{-1}]\otimes V$.
As in \cite{MR2096674}, we define the filtration $V^a(\widetilde{V})$ on $\widetilde{V}$ by
\[
	V^a(\widetilde{V}) = \bigoplus_{k\in\Z}t^kF_{-a+k}(V).
\]
\begin{lem}
We have
\[
	F_{-a}(V)/F_{-a-1}(V)\xrightarrow{\sim}F_{-a}(\widehat{J}_{\Theta,\Pi}(V))/F_{-a-1}(\widehat{J}_{\Theta,\Pi}(V))
	\simeq\bigoplus_{\mu(H) = a}\wtsp_{\mu}(J_{\Theta,\Pi}(V))
\]
\end{lem}
\begin{proof}
We prove that the first homomorphism is isomorphic.
By the definition of $F_{-a-1}(V)$, the homomorphism is injective.
This homomorphism is surjective by Lemma~\ref{lem:wt sp is stable}.
The second homomorphism is obviously an isomorphism.
\end{proof}
From this lemma, if we prove that $V^a(\widetilde{V})$ is a $V$-filtration, then by the description of the nearby cycle functor in terms of $\mathscr{D}$-modules~\cite{MR726425} (see \cite[3]{MR2096674}) and Lemma~\ref{lem:decomp int we sps of J}, we have $\Gamma(G/B,R\psi \widetilde{\mathscr{V}}) = J_{\Theta,\Pi}(V)$.
Hence Theorem~\ref{thm:Emerton-Nadler-Vilonen2} is proved.

To prove that this gives a $V$-filtration, it is sufficient to prove the following lemma. (See \cite[4]{MR2096674}.)
Define a filtration $F_a(U(\overline{\mathfrak{n}}_\Theta))$ by 
\[
	F_a(U(\overline{\mathfrak{n}}_\Theta)) = \bigoplus_{a'\le_\Z a}\wtsp_{H,a'}(U(\overline{\mathfrak{n}}_\Theta)).
\]

\begin{lem}
For $a\in\C$ and $k,l\in\Z$, the following hold.
\begin{enumerate}
\item For a sufficiently large $k$, $F_{-a + k}(V)$ is stable.
\item The module $F_{-a-k}(V)/(F_{-k}(U(\overline{\mathfrak{n}}_\Theta))F_{-a}(V))$ is a finitely generated $U(\mathfrak{m}_\Theta)$-module.
\item For $l\ge 0$, we have $F_{-a-k-l}(V) = F_{-l}(U(\overline{\mathfrak{n}}_\Theta))F_{-a-k}(V)$ for a sufficiently large $k$.
\end{enumerate}
\end{lem}
\begin{proof}
(1)
Take a finite subset $\Lambda\subset\mathfrak{a}_\Theta^*$ such that $\wt_{\mathfrak{a}_\Theta}(\widehat{J}_{\Theta,\Pi}(V))\subset\Lambda - \Z_{\ge 0}\Pi|_{\mathfrak{a}_\Theta}$.
Since $H$ is dominant integral, for all $\mu\in \wt_{\mathfrak{a}_\Theta}(\widehat{J}_{\Theta,\Pi}(V))$ we have $\mu'(H) - \mu(H)\in\Z_{\ge 0}$ for some $\mu'\in\Lambda$.
Take $k$ such that $-a + k\ge \max\{\mu'(H)\mid \mu'\in\Lambda\}$.
For such $k$, $F_{-a + k}(V)$ is stable.

(2, 3)
We can use the same proof as that of \cite[Lemma~2.5]{MR2096674}.
\end{proof}
From this, $V^a(\widetilde{V})$ is a $V$-filtration.
Hence we get Theorem~\ref{thm:Emerton-Nadler-Vilonen2}.
\section{Symmetric space}\label{sec:Symmetric spaces}
Assume that $G$ is of adjoint type.
Let $\omega_\alpha$ be the fundamental coweight for $\alpha\in\Pi$, namely, it is a cocharacter $\omega_\alpha\colon\Cbatu\to G$ which satisfies $\langle \omega_\alpha,\alpha\rangle = 1$ and $\langle \omega_\alpha,\beta\rangle = 0$ for $\beta\in\Pi\setminus\{\alpha\}$.
Since $G$ is of adjoint type, it exists.
Define $\omega\colon(\Cbatu)^\Pi\to A$ by $(t_\alpha)_\alpha\mapsto \prod_{\alpha\in\Pi}\omega_\alpha(t_\alpha)$.
Then $\omega$ gives an isomorphism.

De Concini and Procesi~\cite{MR718125} constructed the wonderful compactification $X$ of $G/K$.
This compactification satisfies the following conditions.
Set $x_0 = K\in G/K$.
\begin{enumerate}
\renewcommand{\theenumi}{C\arabic{enumi}}
\item The variety $X$ is irreducible and proper smooth over $\C$.
\item A $G$-orbit of $X$ is parameterized by a subset of $\Pi$.
We denote the $G$-orbit corresponding to $\Theta\subset\Pi$ by $X_\Theta$.
\item The $G$-orbit $X_\Pi$ is the unique open $G$-orbit and it is isomorphic to $G/K$.
\item The closure of each orbit is smooth.\label{enum:condition, smooth}
\item We have an $\overline{N}$-equivariant open embedding $\overline{N}\times \Caff^\Pi\to X$ such that for all $a = (a_\alpha)\in(\Cbatu)^\Pi$, an element $\omega(a)x_0$ is the image of $(1,(a_\alpha^{-2}))\in\overline{N}\times\Caff^\Pi$.
Moreover, the intersection of $X_\Theta$ and $\overline{N}\times\Caff^\Pi$ is given by $\overline{N}\times (\Cbatu)^\Theta\times\{0\}^{\Pi\setminus\Theta}$.\label{enum:condition, local coodinate}
\item By the above condition, $\overline{N}\times\Caff^\Pi$ is regarded as an open subvariety of $X$.
Then the stabilizer of $(1,(1^\Theta,0^{\Pi\setminus\Theta}))$ in $G$ is $K_\Theta A_\Theta N_\Theta$.\label{enum:condition, stabilizer}
\end{enumerate}
\begin{rem}
The parameterization of $G$-orbits in \cite{MR718125} is different from ours.
In \cite{MR718125}, the open orbit corresponds to $\emptyset$.
\end{rem}

For each $\alpha\in\Pi$, put $Y_\alpha = \overline{X_{\Pi\setminus\{\alpha\}}}$.
By the conditions (\ref{enum:condition, smooth}) and (\ref{enum:condition, local coodinate}),
$\bigcup_{\alpha\in\Pi}Y_\alpha$ is a strict normal crossing divisor.
Let $f\colon \mathcal{X}\to \Caff^\Pi$ be the variety constructed in Section~\ref{sec:Deformation to normal cone} with respect to $\bigcup_{\alpha\in\Pi}Y_\alpha$.
Each $Y_\alpha$ defines the subvariety $\mathcal{Y}_\alpha \subset \mathcal{X}$.
Put $X' = \overline{N}\times\Caff^\Pi$ and regard it as an open subvariety of $X$ by the condition (\ref{enum:condition, local coodinate}).
Then $X'$ defines an open subvariety $\mathcal{X}'$ of $\mathcal{X}$.
Since $Y_\alpha\cap X'$ is isomorphic to $\{(\overline{n},(c_\beta)_{\beta\in\Pi})\mid c_\alpha = 0\}$, $\mathcal{X}'$ is isomorphic to $\overline{N}\times\Caff^\Pi\times\Caff^\Pi$ by Lemma~\ref{lem:local calculation of deformation} (\ref{enum:local calculatio of deformation:description}).

Put $\mathcal{Z} = \mathcal{X}\setminus\bigcup_{\alpha\in\Pi}\mathcal{Y}_\alpha$ and $\mathcal{Z}' = \mathcal{X}'\cap \mathcal{Z}$.
Let $f_\mathcal{Z}\colon \mathcal{Z}\to \Caff^\Pi$.
Then we have $\mathcal{Z}' \simeq\overline{N}\times(\Cbatu)^\Pi\times\Caff^\Pi$.
Define a section $s\colon \Caff^\Pi\to \mathcal{Z}$ of $f_\mathcal{Z}\colon \mathcal{Z}\to\Caff^\Pi$ by $s(t) = (1,1^\Pi,t)\in\overline{N}\times(\Cbatu)^\Pi\times\Caff^\Pi\simeq \mathcal{Z}'\subset\mathcal{Z}$.
For $\Theta\subset \Pi$, set $t_\Theta = (1^\Theta,0^{\Pi\setminus\Theta})$ and $x_\Theta = s(t_\Theta)$.

\begin{lem}\label{lem:action of A on ncd}
For $(a_\alpha)_\alpha\in(\Cbatu)^\Pi$, the action of $\omega(a_\alpha)\in A$ on $\mathcal{X}'\simeq\overline{N}\times\Caff^\Pi\times\Caff^\Pi$ is given by $(\overline{n},d,t)\mapsto (\Ad(\omega(a_\alpha))\overline{n},(a_\alpha^{-2})d,t)$.
\end{lem}
\begin{proof}
This follows from (\ref{enum:condition, local coodinate}) and Lemma~\ref{lem:local calculation of deformation} (\ref{enum:local calculatio of deformation:group action}).
\end{proof}

By Lemma~\ref{lem:deformation to normal cone}, we have $\mathcal{X}_{\Theta} \simeq T_{\overline{X_{\Theta}}}(X)\times(\Cbatu)^{\Theta}$. For each subvariety $\mathcal{W}\subset \mathcal{X}$, 
put $\mathcal{W}_\Theta=\mathcal{W}\cap \mathcal{X}_\Theta$.

\begin{lem}\label{lem:base of Z}
The open subvariety $\mathcal{Z}_\Theta\subset \mathcal{X}_\Theta$ is contained in $T_{X_\Theta}(X)\times (\Cbatu)^\Theta$.
\end{lem}
\begin{proof}
By Lemma~\ref{lem:deformation to normal cone}, we have $\mathcal{X}_{\Theta} \simeq T_{\overline{X_{\Theta}}}(X)\times(\Cbatu)^{\Theta}$ and $(\mathcal{Y}_\alpha)_{\Theta} = T_{Y_\alpha\cap \overline{X_{\Theta}}}(Y_\alpha)\times(\Cbatu)^\Theta$.
For $\alpha\in\Theta$, we have an obvious identity 
$T_{\overline{X_\Theta}}(X)\vert_{Y_\alpha\cap \overline{X_\Theta}}=T_{Y_\alpha\cap \overline{X_{\Theta}}}(Y_\alpha)$.
Thus we have $(\mathcal{X}\setminus\bigcup_{\alpha\in\Theta}\mathcal{Y}_\alpha)_{\Theta} \simeq T_{\overline{X_\Theta}\setminus\bigcup_{\alpha\in\Theta}Y_\alpha}(X)\times (\Cbatu)^{\Theta} = T_{X_\Theta}(X)\times (\Cbatu)^{\Theta}$.
Hence $\mathcal{Z}_\Theta\subset T_{X_\Theta}(X)\times (\Cbatu)^{\Theta}$.
\end{proof}

\begin{lem}\label{lem:stabilizer of x_Theta}
The stabilizer of $x_\Theta$ in $G$ is $K_\Theta N_\Theta$.
\end{lem}
\begin{proof}
Consider the morphism $\mathcal{Z}_\Theta\hookrightarrow T_{X_\Theta}(X)\times (\Cbatu)^{\Theta}\to T_{X_{\Theta}}(X)\to X_\Theta$.
Let $y_\Theta$ be the image of $x_\Theta$.
Then $y_\Theta = (1,(1^\Theta,0^{\Pi\setminus\Theta}))\in\overline{N}\times\Caff^\Pi = X'$ by Lemma~\ref{lem:local calculation of deformation} (\ref{enum:local calculatio of deformation:projection}).
Hence $\Stab_G(x_\Theta)\subset\Stab_G(y_\Theta) = K_\Theta A_\Theta N_\Theta$ by (\ref{enum:condition, stabilizer}).

We prove $K_\Theta N_\Theta\subset \Stab_G(x_\Theta)$.
By Lemma~\ref{lem:deformation to normal cone}, we have $\mathcal{X}_\Pi\simeq X\times(\Cbatu)^\Pi$.
Then $s(t^2)$ is given by $(\omega(t)^{-1}x_0,t^2)\in X\times(\Cbatu)^\Pi$ for $t\in (\Cbatu)^\Pi$
(cf.~Lemma \ref{lem:local calculation of deformation} (1)).
Hence $\Stab_G(s(t^2)) = \Ad(\omega(t)^{-1})K$.
Its Lie algebra is spanned by $\mathfrak{m}$ and $\{\Ad(\omega(t)^{-1})(X + \theta(X))\mid X\in\mathfrak{g}_\beta,\ \beta\in\Sigma^+\}$.
Here, $\mathfrak{g}_\beta$ is the root space for $\beta$.
Since $\Ad(\omega(t)^{-1})(X + \theta(X)) = \beta(\omega(t)^{-1})(X + \beta(\omega(t))^2\theta(X))$, the Lie algebra of $\Stab_G(s(t^2))$ is spanned by $\mathfrak{m}$ and $\{X + \beta(\omega(t))^2\theta(X)\mid X\in\mathfrak{g}_\beta,\beta\in\Sigma^+\}$.
If $\beta = \sum_{\alpha\in\Pi}n_\alpha \alpha$, then $\beta(\omega(t))^2 = \prod_{\alpha\in\Pi}t_\alpha^{2n_\alpha}$ for $t = (t_\alpha)\in(\Cbatu)^\Pi$.
Since this can be extended to any $t\in\Caff^\Pi$, the Lie algebra of $\Stab_G(s(t^2))$ contains the space spanned by $\mathfrak{m}$ and $\{X + \beta(\omega(t))^2\theta(X)\mid X\in\mathfrak{g}_\beta,\beta\in\Sigma^+\}$ for any $t\in\Caff^\Pi$.

Now set $t = t_\Theta$.
Then $s(t^2) = s(t) = x_\Theta$.
For $\beta = \sum_{\alpha\in\Pi}n_\alpha \alpha\in\Sigma^+$, $\beta(\omega(t))^2$ is $0$ or $1$ and it is $1$ if and only if $n_\alpha = 0$ for any $\alpha\in\Pi\setminus \Theta$, namely, $\mathfrak{g}_\beta\subset \mathfrak{m}_\Theta$.
Hence the Lie algebra of $\Stab_G(x_\Theta)$ contains $\Lie(K_\Theta N_\Theta)$.
Therefore, $\Stab_G(x_\Theta)\supset (K_\Theta)^\circ N_\Theta$.
Since $K = MK^\circ$ and $M$ stabilizes $s(t)$ for all $t\in(\Cbatu)^\Pi$ (hence for all $t\in\Caff^\Pi$), we have $\Stab_G(x_\Theta) \supset M(K_\Theta)^\circ N_\Theta = K_\Theta N_\Theta$.

Finally, we prove that $\Stab_{A_\Theta}(x_\Theta)\subset M$.
Since $M\subset K_\Theta$, this implies the lemma.
For $(a_\alpha)\in(\Cbatu)^\Pi$, we have $\omega(a_\alpha)x_\Theta = (1,(a_\alpha^{-2}),t_\Theta)\in\overline{N}\times\Caff^\Pi\times\Caff^\Pi\simeq\mathcal{X}'$.
Hence if $\omega(a_\alpha)\in\Stab_{A_\Theta}(x_\Theta)$, then $a_\alpha^2 = 1$ for all $\alpha\in\Pi$.
Therefore, $\omega(a_\alpha)\in M$.
\end{proof}

\begin{lem}\label{lem:transitive on fiber of Z}
We have $Gx_\Theta = f_\mathcal{Z}^{-1}(t_\Theta)$.
\end{lem}
\begin{proof}
We regard $T_{X_\Theta}(X)$ as a subvariety of $\mathcal{X}$ by
$T_{X_\Theta}(X)=T_{X_\Theta}(X)\times \{1^\Theta\}\subset T_{X_\Theta}(X)\times (\Cbatu)^\Theta\cong \mathcal{X}_\Theta\subset \mathcal{X}$.
By Lemma~\ref{lem:base of Z}, we have $f_\mathcal{Z}^{-1}(t_\Theta) = T_{X_\Theta}(X)\cap \mathcal{Z}$.
Let $p\colon T_{X_\Theta}(X)\to X_\Theta$ be the projection.
Then $p(x_\Theta)=y_\Theta$, where $y_\Theta$ is given in the proof of the previous lemma.
Since $X_\Theta$ is a $G$-orbit, we have $Gy_\Theta = X_\Theta$.
Hence it is sufficient to prove that $\Stab_G(y_\Theta)x_\Theta = p^{-1}(y_\Theta)\cap \mathcal{Z}$.
By (\ref{enum:condition, stabilizer}) and Lemma \ref{lem:stabilizer of x_Theta}, it is equivalent to
showing that $A_\Theta x_\Theta = p^{-1}(y_\Theta)\cap \mathcal{Z}$.
By Lemma \ref{lem:local calculation of deformation} (\ref{enum:local calculatio of deformation:projection}),
we have $p^{-1}(y_\Theta) = \{1\}\times \{1^{\Theta}\}\times \Caff^{\Pi\setminus\Theta}\times\{(1^\Theta,0^{\Pi\setminus\Theta})\}\subset\overline{N}\times(\Cbatu)^\Pi\times\Caff^\Pi = \mathcal{X}'$.
Hence the lemma follows from Lemma~\ref{lem:action of A on ncd}.
\end{proof}
In general, for an algebraic group $H$ and an $H$-variety $Y$, let $\Perv_H(Y)$ be the category of $H$-equivariant perverse sheaves.
Then $\HC_{\Theta,\rho}\simeq\Perv_{K_\Theta N_\Theta}(G/B)\simeq \Perv_{K_\Theta N_\Theta\times B}(G)\simeq\Perv_{B}(G/K_\Theta N_\Theta)$.
Write $\Delta'_\Theta$ for this equivalence.
Let $\Theta_2\subset\Theta_1\subset\Pi$.
Take $n_\alpha\in\Z_{\ge 1}$ for each $\alpha\in\Theta_1\setminus\Theta_2$ and
define $\nu\colon \Caff^1\to \Caff^{\Pi}$ by $\nu(t) = (t^{n_\alpha})_{\alpha\in\Theta_1\setminus\Theta_2}\times (0^{\Pi\setminus\Theta_1})\times (1^{\Theta_2})$.
Put $\mathcal{Z}_\nu = \mathcal{Z}\times_{\Caff^\Pi}\Caff^1$ and
denote the canonical morphism $\mathcal{Z}_\nu \to \Caff^1$ by $f_{\nu}$.
Then, by Lemma \ref{lem:stabilizer of x_Theta} and Lemma \ref{lem:transitive on fiber of Z}, we have $f_{\nu}^{-1}(0)\simeq G/K_{\Theta_2}N_{\Theta_2}$ and $f_{\nu}^{-1}(\Cbatu) \simeq G/K_{\Theta_1}N_{\Theta_1}\times \Cbatu$.
Let $p_\nu\colon f_{\nu}^{-1}(\Cbatu) \simeq G/K_{\Theta_1}N_{\Theta_1}\times \Cbatu\to G/K_{\Theta_1}N_{\Theta_1}$ be the first projection and $R\psi$ be the nearby cycle functor with respect to $f_{\nu}$.
Define $\Kat{\nu}\colon \Perv_{B}(G/K_{\Theta_1}N_{\Theta_1})\to\Perv_B(G/K_{\Theta_2}N_{\Theta_2})$ by $\Kat{\nu} = R\psi\circ p_\nu^*$.

Now we prove the main theorem of this paper.
\begin{thm}
As functors $\HC_{\Theta_1,\rho}\to \Perv_B(G/K_{\Theta_2}N_{\Theta_2})$, we have $\Kat{\nu}\circ \Delta_{\Theta_1}'\simeq \Delta_{\Theta_2}'\circ J_{\Theta_2,\Theta_1}$.
\end{thm}
\begin{proof}
Let $s_\nu\colon \Caff^1 \to \mathcal{Z}_\nu$ be the section of $f_\nu$ obtained by the base change of $s$
under $\nu$.
Consider the following diagram:
\[
\xymatrix{
G/B\ar@{<->}[d]^\simeq_{g\mapsto g^{-1}} & G/B\times\Cbatu\ar[l]^{(g,t)\mapsto \nu(t)g}\ar[r]\ar@{<->}[d]^\simeq & G/B\times \Caff^1\ar@{<->}[d]^\simeq & G/B\ar[l]\ar@{<->}[d]^\simeq\\
B\backslash G & B\backslash G\times\Cbatu\ar[l]\ar[r] & B\backslash G\times \Caff^1 & B\backslash G\ar[l]\\
G\ar[d]\ar[u] & G\times\Cbatu\ar[l]\ar[r]\ar[d]_{(g,t)\mapsto (\omega(\nu(t))^{-1}g,t)}\ar[u] & G\times\Caff^1\ar[d]_{(g,t)\mapsto gs_\nu(t^2)}\ar[u] & G\ar[l]\ar[d]\ar[u]\\
G/K_{\Theta_1}N_{\Theta_1} & G/K_{\Theta_1}N_{\Theta_1}\times\Cbatu\ar[l]\ar[r] & \mathcal{Z}_\nu & G/K_{\Theta_2}N_{\Theta_2}.\ar[l]
}
\]
Every rectangle in the diagram above is cartesian, and every vertical arrow is smooth.
The functor $\Psi_\nu$ of Emerton-Nadler-Vilonen is defined as the nearby cycle functor with
respect to the top row, and the functor $\Kat{\nu}$ is defined as the nearby cycle functor with
respect to the bottom row.
Therefore, we get our theorem by Theorem~\ref{thm:Emerton-Nadler-Vilonen2} and the smooth base change theorem.
\end{proof}

\newcommand{\etalchar}[1]{$^{#1}$}
\def\cprime{$'$} \def\dbar{\leavevmode\hbox to 0pt{\hskip.2ex \accent"16\hss}d}
  \def\Dbar{\leavevmode\lower.6ex\hbox to 0pt{\hskip-.23ex\accent"16\hss}D}
  \def\cftil#1{\ifmmode\setbox7\hbox{$\accent"5E#1$}\else
  \setbox7\hbox{\accent"5E#1}\penalty 10000\relax\fi\raise 1\ht7
  \hbox{\lower1.15ex\hbox to 1\wd7{\hss\accent"7E\hss}}\penalty 10000
  \hskip-1\wd7\penalty 10000\box7}
  \def\cfudot#1{\ifmmode\setbox7\hbox{$\accent"5E#1$}\else
  \setbox7\hbox{\accent"5E#1}\penalty 10000\relax\fi\raise 1\ht7
  \hbox{\raise.1ex\hbox to 1\wd7{\hss.\hss}}\penalty 10000 \hskip-1\wd7\penalty
  10000\box7} \newcommand{\noop}[1]{}

\end{document}